\documentclass{siamart220329}

\usepackage{hyperref}
\usepackage{lipsum}
\usepackage{amsfonts,amsmath,amssymb}
\usepackage{graphicx}
\usepackage{epstopdf}
\usepackage{mathtools}
\usepackage{graphicx}
\usepackage[shortlabels]{enumitem}
\usepackage{algorithm}
\usepackage[end]{algpseudocode}
\usepackage{listings}
\lstdefinestyle{mystyle}{
    basicstyle=\ttfamily\footnotesize,
    breakatwhitespace=false,         
    breaklines=true,                 
    captionpos=b,                    
    keepspaces=true,                 
    showspaces=false,                
    showstringspaces=false,
    showtabs=false,                  
    tabsize=2
}
\ifpdf%
  \DeclareGraphicsExtensions{.eps,.pdf,.png,.jpg}
\else
  \DeclareGraphicsExtensions{.eps}
\fi
\usepackage{amsopn}
\usepackage{booktabs}

\newtheorem{remark}[theorem]{Remark}

\newcommand{\TheTitle}{%
A reduced conjugate gradient basis method for fractional diffusion
}

\newcommand{\TheShortTitle}{%
RCGBM for fractional Laplacian
}

\newcommand{\TheAuthor}{%
  Y. Li, L. Zikatanov, C. Zuo
}

\newcommand{\TheFunding}{The work of Li is partially supported by the Fundamental Research Funds for the Central Universities 226-2023-00039. 
The work of Zikatanov and Zuo is supported in part by NSF DMS-2208249. Zikatanov also acknowledges the support of the U.~S.-Norway Fulbright Foundation.  
}

\headers{\TheShortTitle}{\TheAuthor}
\title{\TheTitle
\thanks{\TheFunding}}

\author{Yuwen Li\thanks{School of Mathematical Sciences, Zhejiang University, Hangzhou, Zhejiang 310058, China 
  (\email{liyuwen@zju.edu.cn}).}
\and Ludmil T. Zikatanov\footnotemark[3]
\and Cheng Zuo\thanks{Department of Mathematics, The Pennsylvania State University,
University Park, PA 16802, USA
  (\email{ltz1@psu.edu}, \email{cxz5180@psu.edu}).}
  }
  
\ifpdf%
\hypersetup{%
  pdftitle={\TheTitle},
  pdfauthor={\TheAuthor}
}
\fi

\begin{document}

\maketitle


\begin{abstract}
This work is on a fast and accurate reduced basis method for solving discretized fractional elliptic partial differential equations (PDEs) of the form $\mathcal{A}^su=f$ by rational approximation. 
A direct computation of the action of such an approximation would require solving multiple (20$\sim$30) large-scale sparse linear systems. Our method constructs the reduced basis using the first few directions obtained from the preconditioned conjugate gradient method applied to one of the linear systems. As shown in the theory and experiments, only a small number of directions (5$\sim$10) are needed to approximately solve all large-scale systems on the reduced basis subspace. This reduces the computational cost dramatically because: (1) We only use one of the large-scale problems to construct the basis; and (2) all large-scale problems restricted to the subspace have much smaller sizes. We test our algorithms for fractional PDEs on a 3d Euclidean domain, a 2d surface, and random combinatorial graphs.  We also use a novel approach to construct the rational approximation for the fractional power function by the  orthogonal greedy algorithm (OGA). 
\end{abstract}

\begin{keywords}
  Fractional PDE, Rational Approximation, Orthogonal Greedy Algorithm, Reduced Basis Method, Conjugate Gradient method, Preconditioning
\end{keywords}

\section{Fractional Diffusion: introduction and preliminaries}\label{sec:intro} Fractional order operators have received extensive attention and have led to many applications in stochastic processes arising from quantum mechanics, biology, economics, see, e.g., \cite{BucurValdinoci2016} and the references therein. 
Recently, fractional operators have also been used to develop parameter-robust preconditioners for complex multi-physics systems (cf.~\cite{BoonKochKuchtaMardal2022,BudisaHu2023}).
     Our focus is on efficient numerical solution of fractional problems of the form: 
      \begin{equation}\label{Asuf}
          \mathcal{A}^su = f, \quad s\in(0,1), 
      \end{equation}
      where $\mathcal{A}$ is a Symmetric and Positive Definite (SPD) operator, $f$ is given and $u$ is the solution. 
      In practice, $\mathcal{A}$ could be an elliptic differential operator on a Euclidean domain, the Laplace-Beltrami operator on a surface, or a weighted graph Laplacian. On bounded domains there are several non-equivalent interpretations for operators of fractional order (cf.~\cite{BonitoPasciak2015,DeliaDuGlusaGunzburgerTianZhou2020,Borthagaray2021}). Our approach in this work utilizes the description for $\mathcal{A}^s$ in terms of the spectral decomposition of $\mathcal{A}$.

      The difficulty in solving \eqref{Asuf} is the  non-local nature of $\mathcal{A}^s$, which results in a dense differential matrix by direct discretizations (cf.~\cite{CelikDuman2012,ZhaoSunHao2014}). 
      As it turns out in \cite{BonitoPasciak2015,HarizanovLazarovMargenovMarinovVutov2018,hofreither2020unified}, etc., rational approximation of $\varphi(z)=z^{-s}$ can reduce the fractional problem \eqref{Asuf} to a series of well-studied SPD problems allowing sparse discretizations. Roughly speaking, the rationale of this idea involves the following two steps: (1) finding the (nearly) best rational approximation $r_n(z)\approx z^{-s}$ with negative poles $\{-t_i\}_{i=1}^n$; (2) construct an efficient solver for the $n$ discrete problems:
\begin{equation}\label{Ahtuf}
          (\mathcal{A}_h+t_i\mathcal{I}_h)u_h(t_i) = f_h, \quad 1\leq i\leq n,
      \end{equation}
where $\mathcal{A}_h$ is a discretization of $\mathcal{A}$, $\mathcal{I}_h$ is the identity, and $u_h(t_i)$ is the numerical (by finite element or finite difference) solution, also known as a snapshot  at $t_i$, in a finite-dimensional space $\mathcal{H}_h$. 
In practice, an accurate rational approximation for $z^{-s}$ has about $n=20\sim30$ poles and thus \eqref{Ahtuf} needs to be solved for at least $20\sim30$ times.  

The main results in this paper are on employing a special reduced basis method to perform the second step by reducing all systems in \eqref{Ahtuf} to a subspace $\mathcal{H}_h^m\subset\mathcal{H}_h$ of dimension $5\sim10$. The savings in computational time are significant as accurate approximation of \eqref{Asuf} leads to $n$ large-scale linear systems in~\eqref{Ahtuf}, typically with $10^6\sim10^9$ unknowns. A key feature of our method is that the cost of constructing $\mathcal{H}_h^m$ is the same as solving only one system of the type \eqref{Ahtuf} by the Preconditioned Conjugate Gradient (PCG). Later in this section, we give more details on the proposed reduced basis approach.  

\subsection{Orthogonal Greedy Algorithm for rational approximation}
The pioneering work by Newman~\cite{1964NewmanDJ-aa} shows that the approximation error of $|z|$ over $[-1,1]$ (equivalently $\sqrt{z}$ on $[0,1]$) by rational functions of degree $(n,n)$ exponentially decays with respect to $n$.
Such a result was further extended to rational approximations of $z^s$ for $s>0$ in \cite{1976TzimbalarioJ-aa,PetrushevPopov1987,Lorentz1996,stahl2003best}. For a comprehensive overview in rational approximation we refer to the monograph by Petrushev and Popov~\cite{PetrushevPopov1987}. 


The classical method for constructing the best uniform rational approximation to a given function is the Remez algorithm (see \cite{PetrushevPopov1987}). Stable implementations of such algorithms require quadruple precision or exact arithmetic. The best uniform rational approximation $r(z)\approx z^{1-s}$ is used in \cite{HarizanovLazarovMargenovMarinovVutov2018,2020HarizanovLazarovMargenovMarinovPasciak-a} to approximate $z^{-s}$ by $r(z)/z$ and to solve fractional diffusion. More recently, it was shown in \cite{nakatsukasa2018aaa} that greedy algorithms based on \emph{rational barycentric interpolation} yield accurate rational approximations for general functions, see \cite{nakatsukasa2018aaa} for the AAA algorithm and \cite{hofreither2021algorithm} for the BRASIL algorithm. 

In this paper, we employ a new strategy for the rational approximation based on the \emph{Orthogonal Greedy Algorithm} (OGA), which is a classical nonlinear algorithm in the field of machine learning, statistics and signal processing, see, e.g., \cite{DavisMallat1997,DeVoreTemlyakov1996,BarronCohenDahmenDeVore2008}. The OGA adaptively selects basis functions $g_1, \ldots,g_n$ from a redundant dictionary $\mathcal{D}$ and uses the orthogonal projection $\varphi_n=P_n\varphi$ onto ${\rm span}\{g_1,\ldots,g_n\}$ to approximate a target function $\varphi$. In the context of solving the fractional problem as~\eqref{Asuf}, we use an OGA with $\varphi(z)=z^{-s}$ and a dictionary 
$\mathcal{D}=\{(z+t)^{-1}\}_{t\in T}$. The resulting rational approximation  is of the form
\begin{equation*}
    z^{-s}\approx r_n(z)=\sum_{i=1}^n\frac{c_i}{z+t_i}.
\end{equation*}
The connection between \eqref{Asuf} and \eqref{Ahtuf} is seen from 
\begin{equation}\label{rationalapproximation}
    \mathcal{A}_h^{-s}f\approx r_n(\mathcal{A}_h)f_h=\sum_{i=1}^nc_i(\mathcal{A}_h+t_i\mathcal{I}_h)^{-1}f_h.
\end{equation}
The OGA directly computes the residues and poles of $r_n$ needed in approximating the action $\mathcal{A}_h^{-s}f_h$. In comparison, algorithms in \cite{nakatsukasa2018aaa,hofreither2021algorithm} output barycentric rational interpolants and then perform a generalized eigenvalue computation that may lead to numerical complex values and loss of accuracy. Moreover, the convergence of the aforementioned OGA can be analyzed rigorously assuming certain regularity condition on $\mathcal{D}$.

\subsection{Reduced Basis Method by PCG}
The Reduced Basis Method (RBM) is  a popular numerical method  for efficiently solving \emph{parametrized} PDEs dated back to 1970s (cf.~\cite{Almroth1978,NoorPeters1980,FinkRheinboldt1983}), see also \cite{Quarteroni2016,HesthavenRozzaStamm2016} for a modern introduction. A standard RBM first constructs an expensive but fixed offline subspace $\mathcal{H}_h^m\subset\mathcal{H}_h$ of small dimension $m$ and then rapidly performs online computation within $\mathcal{H}_h^m$ for many solution instances corresponding to a large number of 
parameters. The RBM could gain significant computational efficiency for many-query parameter-dependent problems provided the size of the parameter set under consideration is large. For instance, the $t$-dependent equation \eqref{Ahtuf} from fractional PDEs is a target problem for the RBM, see, e.g., \cite{DanczulSchoberl2022,DanczulSchoberl2021,DanczulHofreither2022} in this direction. These RBMs, or the equivalent rational Krylov methods (see \cite{DanczulHofreither2022}), essentially seeks approximate solutions of \eqref{Ahtuf} in the snapshot-based offline subspace
\[
\mathcal{H}_h^m={\rm span}\big\{u_h(t_{k_1}), u_h(t_{k_2}),\ldots,u_h(t_{k_m})\big\}.
\]
Clearly, when $m\ll n$, the RBM would be more efficient than directly computing the $n$ solutions of \eqref{Ahtuf}. Considering single-parameter problems, the work  \cite{maday2002priori} is able to  a priori sample snapshot parameters with guaranteed convergence. In general, the offline selection of snapshots is achieved by proper orthogonal decomposition or a posteriori error estimation and greedy algorithms, see, e.g.,  \cite{RozzaHuynhPatera2008,sen2008reduced,binev2011convergence,buffa2012priori,LiSiegel2023,BennerGugercinWillcox2015} for the implementation and theoretical analysis of RBMs.

As far as we know, the mainstream RBMs including recent constructions given in \cite{DanczulSchoberl2022,DanczulHofreither2022} use solution snapshots at a set of parameters to construct the subspace $\mathcal{H}_h^m$. In this work, we propose a novel reduced-basis approach, the \emph{Reduced Conjugate Gradient Basis Method} (RCGBM), to alleviate the computational burden of fractional diffusion. In particular, we run $m$ iterations of the preconditioned conjugate gradient for solving \eqref{Ahtuf} with a fixed parameter $t=d$ and obtain conjugate gradient directions $p_0, p_1, \ldots, p_{m-1}$, where $m$ is small and the preconditioner is designed for $\mathcal{A}_h$. The distinct feature of the RCGBM is the offline subspace 
\[
\mathcal{H}_h^m={\rm span}\big\{p_0, p_1, \ldots, p_{m-1}\big\},
\]
which is built upon conjugate gradient basis for a fixed parameter instead of a set of selected snapshots for multiple parameters used in the classical RBMs. With the help of $\mathcal{H}_h^m$, we then compute the Galerkin projection $u^{{\rm rbm}}_h(t_i)=\mathcal{P}_{\mathcal{H}_h^m}u_h(t_i)$ onto $\mathcal{H}_h^m$ for $1\leq i\leq m$. The cost $\mathcal{O}(m^3)$ of each projection is almost negligible compared with classical fast solvers, e.g., FFT or multigrids, for \eqref{Ahtuf}. The RCGBM finally outputs the numerical solution of \eqref{Asuf} as a superposition of solutions of small-scale subproblems in $\mathcal{H}_h^m$:
\[
u^{{\rm rbm}}_h=\sum_{i=1}^nc_iu_h^{{\rm rbm}}(t_i).
\]
As one can see, the implementation of the RCGBM is as simple as a one-time PCG while the classical RBM needs to solve the $m$ large-scale problems in the form \eqref{Ahtuf} to build $\mathcal{H}_h^m$. We also theoretically show that the convergence of the RCGBM is determined by the quality of the preconditioner for $\mathcal{A}_h$. Furthermore, we show numerically that a classical uniform preconditioner such as the geometric or algebraic multigrid cycle suffices to guarantee the accuracy of RCGBMs. 

An important point we would like to make is that the algorithm proposed here is clearly not restricted to fractional PDEs. In a sense, we consider the rational approximation just as an example of parameter-dependent problems. Our conjugate gradient approach seems applicable to a family of SPD elliptic problems parametrized by the diffusion coefficient. The philosophy in the RCGBM is related to the adaptive Algebraic Multigrid (AMG) in \cite{2013VassilevskiKalchev-aa} which builds subspace hierarchies by re-using the coarse spaces already built for nearby problems. In addition, there exist other interesting interactions between solvers and RBMs, see, e.g., \cite{ChungEfendievLeung2015,Buhr2021} for localized RBMs based on domain decomposition.

The rest of the paper is organized as follows. In Section \ref{sec:oga}, we explain in details the OGA for computing rational approximation of fractional power functions. In Section \ref{sec:rcgbm}, we briefly introduce the PCG and propose the RCGBM for efficient solution of fractional diffusion problems. We then present several numerical tests showing the efficiency of our RCGBM in Section \ref{sec:num}. Section \ref{sec:con} is devoted to concluding remarks.

\section{Orthogonal Greedy Algorithm}\label{sec:oga}
In compressed sensing, statistical regression, and machine learning, greedy algorithms are a class of iterative numerical methods for approximating functions using local optimization at each step, see, e.g., \cite{DeVoreTemlyakov1996,Temlyakov2008} and references therein. Let $\mathcal{V}$ a Hilbert space  and $\mathcal{D}\subset\mathcal{V}$ be a dictionary which is just a collection of elements. For the purpose of constructing good rational approximation for $z^{-s}$, we consider the following orthogonal greedy algorithm developed and analyzed in, e.g., \cite{Pati1993,DeVoreTemlyakov1996,BarronCohenDahmenDeVore2008,JonathanXu2022,LiSiegel2023}:
\begin{equation}\label{OGA0}
    \begin{aligned}
&g_n=\arg\max_{g\in\mathcal{D}}\langle g,\varphi-\varphi_{n-1}\rangle_{\mathcal{V}},\\ &\varphi_n=\mathcal{P}_{\mathcal{H}_n}\varphi,\quad \mathcal{H}_n={\rm span}\big\{g_1, \ldots, g_n\big\},
    \end{aligned}
\end{equation}
where $n\geq0$ and $\varphi_0=0$.
In this algorithm, $\mathcal{P}_{\mathcal{H}_n}$ is the orthogonal projection onto $\mathcal{H}_n$ with respect to the $\mathcal{V}$-inner product $\langle\cdot,\cdot\rangle_{\mathcal{V}}$. The analysis in \cite{stahl2003best} implies that the best uniform rational approximant to $\varphi(z)=z^{-s}$ on $[\varepsilon,1]$ with $\varepsilon>0$ must have simple and negative poles. Motivated by this fact, we first consider the dictionary
\begin{equation}\label{dictionary0}
\mathcal{D}=\left\{\frac{1}{z+t}\right\}_{t\in[b_1,b_2]},\quad 0<b_1<b_2,
\end{equation}
where $[b_1,b_2]$ is a positive interval depending on the spectrum of $\mathcal{A}_h$. The resulting rational approximation is of the form
\begin{equation}\label{rational}
    z^{-s}\approx r_n(z)=\sum_{i=1}^n\frac{c_i}{z+t_i}.
\end{equation}
Here $\{-t_i\}_{i=1}^n$ are the poles of the rational approximant and $\{c_i\}_{i=1}^n$ are the corresponding residues. Under an $L_2$ norm $\|\cdot\|$, 
the accuracy of the full rational approximation scheme \eqref{rationalapproximation} (cf.~\cite{hofreither2020unified}) is given by \begin{equation*}
    \|\mathcal{A}_h^{-s}f_h-r_n(\mathcal{A}_h)f_h\|\leq\|f_h\|\max_{z\in[\lambda_{\min},\lambda_{\max}]}|z^{-s}-r_n(z)|,
\end{equation*}
where $0<\lambda_{\min}<\lambda_{\max}$ are the minimum and maximum eigenvalues of $\mathcal{A}_h$, respectively. In practice, we shall solve the rescaled problem \[(\mathcal{A}_h/\lambda_{\max})^su_h=f_h/\lambda_{\max}^s
\]
and need to construct a good rational approximation for $z^{-s}$ over an interval $[\varepsilon,1]$ with $\varepsilon\leq\lambda^{-1}_{\max}$. To achieve this goal, in \eqref{OGA0} we set $\mathcal{V}=L_2[\varepsilon,1]$ and use the $L_2$ normalized dictionary
\begin{equation}\label{dictionary}
 \mathcal{D}=\mathcal{D}(T)=\left\{\left(\frac{1}{\varepsilon+t}-\frac{1}{1+t}\right)^{-\frac{1}{2}}\frac{1}{z+t}\right\}_{t\in T}
\end{equation}
with a finite $T\subset[b_1,b_2]$.
The specific OGA for approximating $z^{-s}$ is as follows.
\begin{algorithm}[H]
\caption{Orthogonal Greedy Algorithm}\label{a:OGA}
\begin{algorithmic}
\State Input: $s>0$, $\varepsilon>0$, $\mathcal{D}=\mathcal{D}(T)$, and an integer $n$;
\State Initialize: Set $\varphi_0=0$;
\For{$k=1:n$}
\State  $g_k=\arg\max_{g\in\mathcal{D}}\int_\varepsilon^1g(z)\big(z^{-s}-\varphi_{k-1}(z)\big)dz;$
\State Find $\varphi_k\in\mathcal{H}_k:={\rm span}\big\{g_1,\ldots,g_k\big\}$ s.t. 
\[
\int_\varepsilon^1\varphi_k(z)g(z)dz=\int_\varepsilon^1z^{-s}g(z)dz\quad\text{ for all }g\in\mathcal{H}_k;
\] 
\EndFor
\end{algorithmic}
\end{algorithm}

Convergence analysis of the OGA is an important topic and has been analyzed by many researchers. Using the following norm and space (see \cite{DeVoreTemlyakov1996,BarronCohenDahmenDeVore2008})
\begin{align*}
    \|\varphi\|_{\mathcal{L}_1(\mathcal{D})}&:=\inf\left\{\sum_i|c_i|: \varphi=\sum_{g_i\in\mathcal{D}}c_ig_i\right\},\\
    \mathcal{L}_1(\mathcal{D})&:=\big\{\varphi\in\mathcal{V}: \|\varphi\|_{\mathcal{L}_1(\mathcal{D})}<\infty\big\},
\end{align*}
the classical error estimate of the OGA for $\varphi(z)=z^{-s}$ (see \cite{DeVoreTemlyakov1996,BarronCohenDahmenDeVore2008}) reads 
\begin{equation}\label{OGAestimate0}
    \|\varphi-\varphi_n\|_{L_2[\varepsilon,1]}\leq \|\varphi\|_{\mathcal{L}_1(\mathcal{D})}(n+1)^{-\frac{1}{2}}.
\end{equation}

Recently, such an estimate has been improved provided the metric entropy number $\varepsilon_n({\rm co}(\mathcal{D}))$ of the symmetric convex hull of $\mathcal{D}$ is asymptotically small, see \cite{JonathanXu2022,LiSiegel2023}. For example, the work \cite{LiSiegel2023} has shown that 
\begin{equation}\label{OGAestimate}
    \|\varphi-\varphi_n\|_{L_2[\varepsilon,1]}\leq \frac{(n!V_n)^\frac{1}{n}}{\sqrt{n}}\|\varphi\|_{\mathcal{L}_1(\mathcal{D})}\varepsilon_n({\rm co}(\mathcal{D})),
\end{equation}
where $V_n$ is the volume of the $n$-dimensional unit ball. For example, a wide class of dictionaries with $s$-order bounded derivatives in $\mathbb{R}^d$ satisfies $\varepsilon_n({\rm co}(\mathcal{D}))=\mathcal{O}(n^{-\frac{1}{2}-\frac{s}{d}})$, see \cite{JonathanXuFoCM}. Therefore, the OGA in our case is expected to enjoy   convergence much faster than $\mathcal{O}(n^{-\frac{1}{2}})$.

It is noted that \eqref{OGAestimate0} and \eqref{OGAestimate} hinge on the membership of $\varphi\in\mathcal{L}_1(\mathcal{D})$, which will be investigated and addressed in a separate paper. We postpone the implementation details and the code of Algorithm \ref{a:OGA} with $\varphi(z)=z^{-s}$ in 
Section \ref{sec:num} and Appendix.

\section{Reduced Conjugate Gradient Basis Method}\label{sec:rcgbm}
If $\mathcal{A}=-\Delta$ is the negative Laplacian or a more general second-order differential operator, then  $\mathcal{A}_h$ arising from the finite-element or finite-difference discretization often corresponds to a  sparse and large-scale ill-conditioned SPD matrix of millions or billions of rows and columns. In such cases, it would be time-consuming to solve the subproblems in \eqref{Ahtuf} for all $1\leq i\leq n$ even in the presence of an optimal fast solver, e.g., the multigrid-preconditioned conjugate gradient. 
The RCGBM in this work is invented to further significantly reduce the computational cost  below that of calling multigrid-PCG $n$ times and meanwhile to maintain numerical accuracy at the same level. 

\subsection{Preconditioned Conjugate Gradient}\label{sec:pcg}
The PCG proposed by Hestenes and Stiefel \cite{hestenes1952methods} is one of the most successful (iterative) methods for solving sparse systems of linear equations developed in the 20th century. For simplicity of presentation, we include the PCG algorithm for $\mathbb{A}\mathbf{x}=\mathbf{b}\in\mathbb{R}^N$ with initial guess zero, where $\mathbb{A}$ is an SPD matrix defining the $\mathbb{A}$-norm $\|\mathbf{x}\|_\mathbb{A}:=\sqrt{\mathbf{x}^\top\mathbb{A}\mathbf{x}}$.
\begin{algorithm}[H]
\caption{Preconditioned  Conjugate Gradient}\label{a:PCG}
\begin{algorithmic}
\State Input: $\mathbf{x}_0\in\mathbb{R}^N$, an SPD matrix $\mathbb{B}\in\mathbb{R}^{N\times N}$, an integer $m\geq1$;
\State Initialize:  
$\mathbf{r}_0 = \mathbf{b} - \mathbb{A}\mathbf{x}_0$, $\mathbf{p}_0=\mathbf{z}_0=\mathbb{B}\mathbf{b}$;
\For{$k=1:m-1$}
\State $\alpha_k=\mathbf{r}_{k-1}^\top \mathbf{z}_{k-1}/\mathbf{p}^\top_{k-1}\mathbb{A}\mathbf{p}_{k-1}$;
\State $\mathbf{x}_k =  \mathbf{x}_{k-1} + \alpha_k \mathbf{p}_{k-1}$;
\State $\mathbf{r}_k =  \mathbf{r}_{k-1} - \alpha_k\mathbb{A}\mathbf{p}_{k-1}$; 
\State $\mathbf{z}_k =  \mathbb{B} \mathbf{r}_k$;
\State $\beta_k =  r_k^\top \mathbf{z}_k/\mathbf{r}_{k-1}^\top\mathbf{z}_{k-1}$;
\State $\mathbf{p}_k = \mathbf{z}_k + \beta_k\mathbf{p}_{k-1}$; 
\EndFor
\end{algorithmic}
\end{algorithm}
The next theorem is a classical result addressing the convergence rate of the PCG under the $\mathbb{A}$-norm, see, e.g., \cite{Xu1992,Braess2007}.
\begin{theorem}\label{PCGerror}
Let $\kappa(\mathbb{BA})=\lambda_{\max}(\mathbb{BA})/\lambda_{\min}(\mathbb{BA})$ be the spectral condition number of $\mathbb{BA}$. The PCG algorithm \ref{a:PCG} satisfies
    \begin{equation}\label{PCGconvergence}
    \|\mathbf{x} - \mathbf{x}_m \|_\mathbb{A}
    \leq
    2\left(\frac{\sqrt{\kappa(\mathbb{BA})}-1}{\sqrt{\kappa(\mathbb{BA})+1}}\right)^m
    \|\mathbf{x} - \mathbf{x}_0\|_\mathbb{A}.
    \end{equation}
  \end{theorem}
We shall show that the PCG, as an alternative to the greedy algorithm or proper orthogonal decomposition used in the existing reduced basis method (cf.~\cite{Maday2006,RozzaHuynhPatera2008,binev2011convergence,Quarteroni2016}), can provide a good reduced basis.

\subsection{The RCGBM Algorithm}\label{sec:main}
As mentioned before, our main goal is to efficiently solve the discrete problem \eqref{Ahtuf} with a number of varying parameters $t_i$. 
To be more precise, we assume that $\mathcal{H}_h$ admits an $L_2$-type inner product $(\cdot,\cdot)$ and $\mathcal{A}_h$ is determined by a bilinear form $a: \mathcal{H}_h\times\mathcal{H}_h\rightarrow\mathbb{R}$ as follows
\begin{equation}
    a(v_h,w_h)=(\mathcal{A}_hv_h,w_h),\quad v_h, w_h\in\mathcal{H}_h.
\end{equation}
Here $a(\cdot,\cdot)$ is assumed to be SPD in the sense that for $v_h, w_h\in\mathcal{H}_h$,
\begin{equation*}
  a(v_h,w_h)=(w_h,v_h),
\quad\mbox{and}\quad  a(v_h,v_h)\geq a_0(v_h,v_h),
\end{equation*} 
where $a_0>0$ is a constant.
The variational formulation of \eqref{Ahtuf} reads
\begin{equation}\label{Ahtufvar}
    a(u_h(t_i),v_h)+t_i(u_h(t_i),v_h)= (f_h,v_h),\quad v_h\in\mathcal{H}_h.
\end{equation}
We recall that 
Algorithm \ref{a:OGA}  outputs the partial fraction decomposition  
\begin{equation*}
  r_n(z) = \sum_{i=1}^n \frac{c_i}{z + t_i}\approx z^{-s}.
\end{equation*}
Then instead of solving the fractional problem $\tilde{u}_h=\mathcal{A}_h^{-s} f_h$ directly, an approximate solution is obtained by
\begin{equation}\label{eq:ur}
   u_h=r_n(\mathcal{A}) f_h = \sum_{i=1}^n c_i (\mathcal{A}_h+t_i \mathcal{I}_h)^{-1} f_h\approx\tilde{u}_h.
\end{equation} 

Now we use a PCG to solve a sample problem
\begin{equation*}
\mathcal{A}_du_h=(\mathcal{A}_h + d\mathcal{I}_h)u_h = f_h,   
\end{equation*}
where $d\geq0$ is a fixed constant. In our RCGBM, the resulting conjugate gradient directions $p_0, p_1,\ldots, p_{m-1}$ serve as a reduced basis. We remark that the quality of this reduced basis is almost independent of  the particular value of $d$. To avoid trivial termination of the PCG iteration, the preconditioner $\mathcal{B}_h$ in the PCG algorithm is set to be the exact inverse or a multigrid cycle for 
$$\mathcal{A}_h+s\mathcal{I}_h,\quad 0\leq s\neq d.$$
The key observation is that each inverse action in \eqref{eq:ur} can be approximated in the tiny subspace $\mathcal{H}^m_h = \operatorname{span}\big\{p_0, p_1,\dots, p_{m-1}\big\}$. We summarize the aforementioned strategy in the next algorithm. 
\begin{algorithm}[H]
  \caption{RCGBM($s, d, m$)\label{a:RCGBM}}
  \begin{algorithmic}
    \State Input: Distinct non-negative constants $s\neq d$, and an integer $m>0$;
    \State Preconditioning: Construct a preconditioner $\mathcal{B}_h: \mathcal{H}_h\rightarrow\mathcal{H}_h$ for $\mathcal{A}_h+s\mathcal{I}_h$;
    \State Initialize: $u_0=0$,
    $r_0=f_h$ and $p_0=z_0=\mathcal{B}_hr_0;$ 
   \For{$k=1:m-1$}
   \State $\alpha_k=(r_{k-1},z_{k-1})/(\mathcal{A}_dp_{k-1},p_{k-1})$;
   \State $u_k =  u_{k-1} + \alpha_k p_{k-1}$; 
   \State $r_k =  r_{k-1} - \alpha_k \mathcal{A}_dp_{k-1}$; 
   \State $z_{k} =  \mathcal{B}_h r_{k}$; 
   \State $\beta_k = (r_{k},z_{k})/(r_{k-1},z_{k-1})$;
   \State $p_{k} = z_{k} + \beta_k p_{k-1}$; 
   \EndFor
   \State Set the subspace:  $\mathcal{H}_h^m = {\rm span}\big\{p_0,\dots,p_{m-1}\big\}$;
   \State Solve subproblems: Find $u_h^{{\rm rbm}}(t_i):=\mathcal{P}_{\mathcal{H}_h^m}u_h(t_i)\in\mathcal{H}_h^m$, $i=1,2,\ldots,n,$ s.t. \[
   a(u_h^{{\rm rbm}}(t_i),v_h)+(u_h^{{\rm rbm}}(t_i),v_h)=(f_h,v_h),\quad v_h\in\mathcal{H}_h^m;
   \]
   \State RCGBM solution: $u_h^{{\rm rbm}} = \sum_{i=1}^nc_iu_h^{{\rm rbm}}(t_i)$. 
  \end{algorithmic}  
\end{algorithm}
The main computational cost in Algorithm \ref{a:RCGBM} is attributed to the construction of $\mathcal{B}_h$ and the for-loop, which 
 is simply the PCG iteration. The solution processes for $\big\{u_h^{\rm rbm}(t_i)\big\}_{1\leq i\leq n}$ is $\mathcal{O}(nm^3)$ and is almost negligible.

Next we rewrite Algorithm \ref{a:RCGBM} in matrix notation when  solving fractional diffusion by finite element methods. Let $\Omega\subset\mathbb{R}^d$ be a Lipschitz domain, and $\alpha, \gamma$ be sufficiently regular functions on $\Omega$ with $\min_{x\in\Omega}\alpha(x)>0$, $\min_{x\in\Omega}\gamma(x)\geq0$. We consider the following fractional problem
\begin{equation}\label{fractionalPDE}
    \mathcal{A}^su=(-\nabla\cdot(\alpha\nabla) +\gamma\mathcal{I})^su=f\quad\text{on}\quad\Omega
\end{equation}
under the homogeneous boundary condition $u|_{\partial\Omega}=0.$ Given a triangulation $\mathcal{T}_h$ of $\Omega$, let $\mathcal{H}_h\subset H_0^1(\Omega)$ be the continuous and piecewise linear finite element space with the nodal basis $\{\phi_i\}_{1\leq i\leq N}$. In this case,  $\mathcal{A}_h: \mathcal{H}_h\rightarrow\mathcal{H}_h$ is determined by 
\begin{equation*}
\begin{aligned}
      &(\mathcal{A}_hv_h,w_h)=(\mathcal{A}_hv_h,w_h)_{L_2(\Omega)}\\
      &=a(v_h,w_h)=\int_\Omega\big(\alpha\nabla v_h\cdot\nabla w_h+\gamma v_hw_h \big)dx,\quad v_h, w_h\in\mathcal{H}_h.  
\end{aligned}
\end{equation*}
The finite-element stiffness and mass matrices read
\begin{equation*}
    \mathbb{A}:=\big(a(\phi_j,\phi_i)\big)_{1\leq i,j\leq N},\quad
    \mathbb{M}:=\big((\phi_j,\phi_i)\big)_{1\leq i,j\leq N},\quad
    \mathbb{A}_d:=\mathbb{A}+d\mathbb{M}.
\end{equation*}
In coordinate form, we have $u_h(t)=\sum_{i=1}^NU_i(t)\phi_i$ and
\begin{align*}
     \mathbf{u}(t):=\big(U_1(t),\ldots,U_N(t)\big)^\top,\quad \mathbf{f}:=\big((f,\phi_1),\ldots,(f,\phi_N)\big)^\top.
\end{align*}
Then \eqref{Ahtufvar} is translated into
\begin{equation*}
    (\mathbb{A}+t_i\mathbb{M})\mathbf{u}(t_i)=\mathbf{f}.
\end{equation*}
Let $\mathbf{p}_j$ be the coordinate vector of $p_j$ with respect to $\{\phi_i\}_{1\leq i\leq N}$. 
The matrix-vector formulation of Algorithm \ref{a:RCGBM} is given below.
\begin{algorithm}[H]
  \caption{RCGBM($s, d, m$) in matrix form}\label{a:RCGBMmatrix}
  \begin{algorithmic}
    \State Input: Distinct non-negative constants $s\neq d$, and an integer $m>0$;
    \State Preconditioning: Constrcut a preconditioner $\mathbb{B}\in \mathbb{R}^{N\times N}$ for $\mathbb{A}+s\mathbb{M}$;
    \State Initialize: 
    $\mathbf{r}_0=\mathbf{f}_h$, $\mathbf{z}_0=\mathbb{B}\mathbf{r}_0$, $\mathbf{p}_0=\mathbf{z}_0$, $l_0=\mathbf{p}_0^\top\mathbb{A}_d\mathbf{p}_0$; 
   \For{$k=1:m-1$}
   \State $\alpha_k=\mathbf{r}_{k-1}^\top\mathbf{z}_{k-1}/l_{k-1}$;
   \State $\mathbf{u}_k = \mathbf{u}_{k-1} + \alpha_k \mathbf{p}_{k-1}$; 
   \State $\mathbf{r}_k =  \mathbf{r}_{k-1} - \alpha_k \mathbb{A}_d\mathbf{p}_{k-1}$; 
   \State $\mathbf{z}_k =  \mathbb{B}\mathbf{r}_k$; 
   \State $\beta_k =  \mathbf{r}_k^\top \mathbf{z}_k/\mathbf{r}_{k-1}^\top\mathbf{z}_{k-1}$;
   \State $\mathbf{p}_k = \mathbf{z}_k + \beta_k\mathbf{p}_{k-1}$; 
   \State $l_k=\mathbf{p}_k^\top\mathbb{A}_d\mathbf{p}_k$;
   \EndFor
   \State Set the projection matrix: $\mathbb{P} = [\mathbf{p}_0/l_0^\frac{1}{2},\ldots,\mathbf{p}_{m-1}/l_{m-1}^\frac{1}{2}]$;
   \State Solve subproblems:  $\mathbf{u}^{\rm rbm}(t_i)=\mathbb{P}(\mathbb{P}^\top\mathbb{A}\mathbb{P}+t_i\mathbb{P}^\top\mathbb{M}\mathbb{P})^{-1}\mathbb{P}^\top\mathbf{f}$, $1\leq i\leq n$;
   \State RCGBM solution: $\mathbf{u}^{\rm rbm} = \sum_{i=1}^nc_i\mathbf{u}^{\rm rbm}(t_i)$.
  \end{algorithmic}  
\end{algorithm}
After finishing the PCG for-loop of Algorithm \ref{a:RCGBMmatrix}, we normalized the conjugate gradient basis $\{\mathbf{p}_k\}_{k=0}^{m-1}$ under the $\|\cdot\|_{\mathbb{A}_d}$-norm. To ensure the advantage of the RCGBM, the matrix-matrix and matrix-vector multiplications  $\mathbb{P}^\top\mathbb{A}\mathbb{P}$, $\mathbb{P}^\top\mathbb{M}\mathbb{P}$, $\mathbb{P}^\top\mathbf{f}$ should be pre-computed and stored prior to solving for $\mathbf{u}^{\rm rbm}(t_i)$ for any $i$. Then the cost of the direct inversion in $(\mathbb{P}^\top\mathbb{A}\mathbb{P}+t_i\mathbb{P}^\top\mathbb{M}\mathbb{P})^{-1}\mathbb{P}^\top\mathbf{f}$ is $\mathcal{O}(m^3)$.
\begin{remark}
Algorithms \ref{a:RCGBM} and \ref{a:RCGBMmatrix} could also be applied to finite-difference discretizations of fractional diffusion. In such cases, the numerical solution $u_h(t_i)$ is a vector formed by approximate values of $u(t_i)$ at finite difference points $\{x_j\}_{1\leq j\leq N}$, and $(\cdot,\cdot)$ becomes the discrete $\ell_2$ inner product in $\mathcal{H}_h=\mathbb{R}^N$. The corresponding mass matrix $\mathbb{M}$ then reduces to the trivial identity matrix $\mathcal{I}_h=\mathbb{I}$, and the stiffness matrix $\mathbb{A}=\mathcal{A}_h$ is simply a finite difference approximation to $-\nabla\cdot(\alpha\nabla) +\gamma\mathcal{I}$.
\end{remark}

By the well-known property of the conjugate gradient method, we have that column of $\mathbb{P}$ is an  orthonormal basis under the $\mathbb{A}_d$-inner product and thus the RCGBM is well-defined. We refer to the classical texts \cite{hestenes1952methods,Braess2007} for a detailed explanation.
\begin{proposition}\label{prop: p-A-orthogonal}
The set of vectors $\{\mathbf{p}_k\}_{k=0}^{m-1}$ in Algorithm \ref{a:RCGBM} is orthonormal with respect to the $\mathbb{A}_d$-inner product, i.e. 
  \begin{subequations}
  \begin{align}
          \mathbf{p}_i^\top\mathbb{A}_d\mathbf{p}_j &= 0, \quad0\leq i\neq j\leq m-1,\\
                \mathbf{p}_i^\top\mathbb{A}_d\mathbf{p}_i &= 1,\quad 0\leq i\leq m-1.
  \end{align}
  \end{subequations}
\end{proposition}

Based on our numerical experiments, typically, only a small number of PCG directions (5$\sim$10) are needed to approximate the actions $(\mathcal{A}_h + t_i \mathcal{I}_h)^{-1}f_h$. The next theorem theoretically confirms the convergence of the RCGBM measured by the norm $\|v_h\|_{\mathcal{A}_h}^2:=(\mathcal{A}_hv_h,v_h)$.
\begin{theorem}\label{errorRCGBM}
Let $\mathcal{A}_h^i=\mathcal{A}_h+t_i\mathcal{I}_h$ and $u_h=\sum_{i=1}^nc_i(\mathcal{A}_h+t_i\mathcal{I}_h)^{-1}f_h$. For the RCGBM(s,d,m) with $\mathcal{B}_h=(\mathcal{A}_h+s\mathcal{I}_h)^{-1}$ and $s\not\in\{t_i\}_{1\leq i\leq n}$,  we have
  \begin{equation}\label{RBMconvergence}
    \|u_h - u_h^{{\rm rbm}}\|_{\mathcal{A}_h}\leq 2\sum_{i=1}^n|c_i| \left(\frac{\sqrt{\kappa(\mathcal{B}_h\mathcal{A}_h^i)} - 1}{\sqrt{\kappa(\mathcal{B}_h\mathcal{A}_h^i)} + 1}\right)^m\|u_h(t_i)\|_{\mathcal{A}_h^i}.
  \end{equation}
\end{theorem}
\begin{proof}
Applying the PCG to $\mathcal{B}_h\mathcal{A}_h^iu_h(t_i)= \mathcal{B}_hf_h$ yields the conjugate gradient directions $p_0^i, p_1^i, \ldots, p_{m-1}^i$ and the subspace \[
\mathcal{H}_{h,i}^m={\rm span}\big\{p_0^i, p_1^i, \ldots, p_{m-1}^i\big\}.
\]
It follows from $\mathcal{B}_h\mathcal{A}_h^i = \mathcal{I} + (t_i-s)\mathcal{B}_h$ that 
\begin{equation*}
    \begin{aligned}
    \mathcal{H}_{h,i}^m&= \operatorname{span}\big\{\mathcal{B}_hr_0, \mathcal{B}_h\mathcal{A}_h^i\mathcal{B}_hr_0, \ldots,(\mathcal{B}_h\mathcal{A}_h^i)^{m-1} \mathcal{B}_hr_0\big\}\\
    &= {\rm span}\big\{\mathcal{B}_hr_0, \mathcal{B}^2_hr_0, \ldots,B_h^mr_0\big\} = \mathcal{H}_h^m,
    \end{aligned}
\end{equation*} 
where $\mathcal{H}_h^m$ is the reduced basis subspace used in the RCGBM. For the subproblems in the RCGBM, we
note that computing the Galerkin projection \[
\mathcal{P}_{\mathcal{H}_{h,i}^m}u_h(t_i)=\mathcal{P}_{\mathcal{H}_{h}^m}u_h(t_i)
\]
under the $\mathcal{A}_h^i$ inner product is equivalent to solving $\mathcal{A}_h^iu_h(t_i)= f_h$ by $m$ PCG iterations. Therefore using Theorem \ref{PCGerror} we obtain
\begin{equation}\label{errori}
\begin{aligned}
    &\|u_h(t_i)-u^{\rm rbm}_h(t_i)\|_{\mathcal{A}_h}\leq\|u_h(t_i)-u^{\rm rbm}_h(t_i)\|_{\mathcal{A}^i_h}\\
      &\quad\leq2\beta_i^m\|u_h(t_i) - 0\|_{\mathcal{A}^i_h} = 2\beta_i^m \|u_h(t_i)\|_{\mathcal{A}_h^i},
\end{aligned}
\end{equation}
where the contraction factor is $$\beta_i=\frac{\sqrt{\kappa(\mathcal{B}_h\mathcal{A}_h^i)} - 1}{\sqrt{\kappa(\mathcal{B}_h\mathcal{A}_h^i)} + 1}.$$ 
Finally combining 
$u_h^{\rm rbm}=\sum_{i=1}^nc_iu_h^{\rm rbm}(t_i)$ with \eqref{errori} leads to
  \begin{equation*}
  \begin{aligned}
    \|u_h-u_h^{\rm rbm}\|_{\mathcal{A}_h}&= \left\|\sum_{i=1}^k c_i \big(u_h(t_i)-u_h^{\rm rbm}(t_i)\big)\right\|_{\mathcal{A}_h}\\
    &\leq2\sum_{i=1}^n |c_i|\beta_i^m\|u_h(t_i)\|_{\mathcal{A}^i_h}.
  \end{aligned}
  \end{equation*}
The proof is complete.
\end{proof}
The above estimate may be pessimistic in some cases due to the convergence  \eqref{PCGconvergence} is generally not optimal. We refer to  \cite{Hackbusch1994,GrahamHagger1999,XuZhu2008} for sharper convergence rates of the PCG provided the  ``bad'' eigenvalues of $\mathcal{B}_h\mathcal{A}_h^i$ are clustered.

For large-scale problems, one may replace the exact inverse $\mathcal{B}_h=(\mathcal{A}_h+s\mathcal{I}_h)^{-1}$ with any practical preconditioner such as the geometric or algebraic multigrid for $\mathcal{A}_h$. The resulting practical version of the RCGBM turns out to still work pretty well. Our analysis shows that the error of the RCGBM converges exponentially with respect to the number $m$ of  the reduced basis vectors used in the offline space $\mathcal{H}_h^m$.

To further simplify the upper bound in \eqref{RBMconvergence}, we apply Theorem \ref{errorRCGBM} with $s=0$ to the fractional diffusion model \eqref{fractionalPDE}. It follows from the Poincar\'e inequality $\|v_h\|\leq C_P\|v_h\|_{\mathcal{A}_h}$ that
\begin{equation*}
    (\mathcal{A}_hv_h,v_h)\leq(\mathcal{A}_h^iv_h,v_h)\leq(1+t_iC_P^2)(\mathcal{A}_hv_h,v_h).
\end{equation*}
As a result, we obtain the bound for the condition number 
\begin{equation*}
   \kappa(\mathcal{B}_h\mathcal{A}_h^i)\leq\kappa_i:=(1+t_iC_P^2).
\end{equation*}
Combining this estimate with the elliptic stability $\|u_h(t_i)\|_{\mathcal{A}_h^i}\leq C_\Omega\|f\|_{H^{-1}(\Omega)}$ yields
  \begin{equation}
    \|u_h - u_h^{{\rm rbm}}\|_{\mathcal{A}_h} \leq C_\Omega\sum_{i=1}^n|c_i| \left(\frac{\sqrt{\kappa_i} - 1}{\sqrt{\kappa_i} + 1}\right)^m\|f\|_{H^{-1}(\Omega)}.
  \end{equation}

\section{Numerical Experiments}\label{sec:num} 
To show the efficiency of the proposed method, we present numerical experiments for the fractional diffusion $$\mathcal{A}^{\frac{1}{2}}u=f$$ 
on a ``domain" $\Omega$ that will be specified later. As mentioned in Section \ref{sec:oga}, in each experiment we shall devise an upper bound $\Lambda\geq\lambda_{\max}$ for the maximum eigenvalue $\lambda_{\max}$ of the discrete operator $\mathcal{A}_h$ and practically solve the rescaled discrete problem 
\begin{equation}\label{rescaled}
    (\mathcal{A}_h/\Lambda)^{\frac{1}{2}}u_h=f_h/\Lambda^\frac{1}{2}.
\end{equation}
To implement the OGA, we need to use a dictionary $\mathcal{D}(T)$ in \eqref{dictionary} with a finite discrete set $T$. It is noted that the OGA theoretically reduces the $L_2$ approximation error and does not control the max norm error well near the left boundary $\varepsilon\ll1$. This phenomenon is partly due to the fact that the target function $z^{-s}$ goes to infinity at the origin.  To resolve the singularity, we partition $[\varepsilon,1]$ by a non-uniform mesh that is graded towards 0. Meanwhile, members of the set $T$ are also graded toward 0. Moreover, we use Algorithm \ref{a:OGA} with  $\varepsilon=10^{-8}$ to uniformly approximates $z^{-s}$ 
on $[10^{-6},1]$, where the \texttt{MATLAB} or \texttt{GNU Octave} code is listed in the appendix. The residues $c_i$ and poles $-t_i$ of the rational approximant used in the RCGBM are also shown in Table \ref{tab:OGArespol} in the appendix.

\subsection{Fractional Laplacian on the Cube}\label{subsec:cube}
The first experiment is concerned with the unit cube $\Omega=[0,1]^3$ and the negative Laplacian $\mathcal{A}=-\Delta$ on $\Omega$. We take the homogeneous Dirichlet boundary condition and use 
\begin{align*}
  u(x,y,z)=\sin(\pi x)\sin(\pi y)\sin(\pi z)
\end{align*}
as the exact solution. By the spectral definition of fractional Laplacian and the relation $-\Delta u=3\pi^2u$, the corresponding right hand side is
\[
f=(-\Delta)^{\frac{1}{2}}u=(3\pi^2)^\frac{1}{2}u.
\]

We start with the partition $\mathcal{T}_0$ of $\Omega$ by six tetrahedra and continue constructing triangulation $\mathcal{T}_\ell$ with successive regular refinement in the sense of \cite{Bey2000}, that is, dividing each tetrahedral element into eight sub-elements. The fully faithful rational approximation solution $u_h=\sum_{i=1}^nc_iu_h(t_i)$ for $(-\Delta)^\frac{1}{2}u=f$ is computed in the continuous and piecewise linear finite-element space $\mathcal{H}_h$ on $\mathcal{T}_h=\mathcal{T}_L$. In Algorithm \ref{a:RCGBMmatrix}, the preconditioner $\mathbb{B}$ is a classical V-cycle multigrid based on the nested grid hierarchy $\{\mathcal{T}_\ell\}_{0\leq\ell\leq L}$. The RCGBM(0,1,$m$) is then applied to the rescaled problem \eqref{rescaled},
where $\Lambda=20/h^2$ is used as an upper bound of  $\lambda_{\max}=\mathcal{O}(h^{-2})$ with $h^3$ being the volume of the tetrahedron in the current mesh.
Numerical results are presented in Figure \ref{fig:errorCube} and Table \ref{tab:errorCube} with $\|\cdot\|=\|\cdot\|_{L_2(\Omega)}$ and $NV$ the number of grid vertices in $\mathcal{T}_h$. 

It is observed from Figure \ref{fig:errorCube} that as the number $m$ of conjugate gradient basis grows, the error $\|u_h(t_i)-u^{{\rm rbm}}_h(t_i)\|_{L_2(\Omega)}$ is decreasing. Moreover, a small $5\leq m\leq 10$ suffices to control the reduces basis error within reasonable accuracy. Table \ref{tab:errorCube} shows that the overall $L_2$ accuracy of the reduced basis solution $u_h^{{\rm rbm}}$ and the high-fidelity $u_h$ are almost the same, justifying the significant efficiency of the RCGBM. We remark that the RCGBM only needs to call the multigrid once while the fully faithful rational finite element method has to call the multigrid preconditioner 20 times. The computational cost of the latter still remains high for the largest 3d problem with 2146689 grid vertices, although the multigrid is known as the optimal solver for elliptic problems.

In addition, we also compute the $H^1$ error between the nodal interpolant $u_I$ (of $u$) and $u_h^{{\rm rbm}}, u_h$. It is not surprising to observe that $\|\nabla(u_I-u_h)\|$ exhibits second-order superconvergence property, see, e.g., \cite{BankXu2003a,ChenLong2006,Li2018SINUM} for recovery-type global superconvergence of finite elements for standard elliptic equations. An interesting observation is that the RCGBM solution is able to preserve such superconvergence, which indicates that a postprocessing procedure $R_h$, e.g., the gradient recovery (cf.~\cite{BankXu2003a,ZhangNaga2005,HuangYi2010,BankLi2019}), would lead to second-order superconvergence of $\|\nabla u - R_h\nabla u_h^{{\rm rbm}}\|$ for the linear element. 
  \begin{figure}[H]
    \centering
\includegraphics[width=10cm,height=6cm]{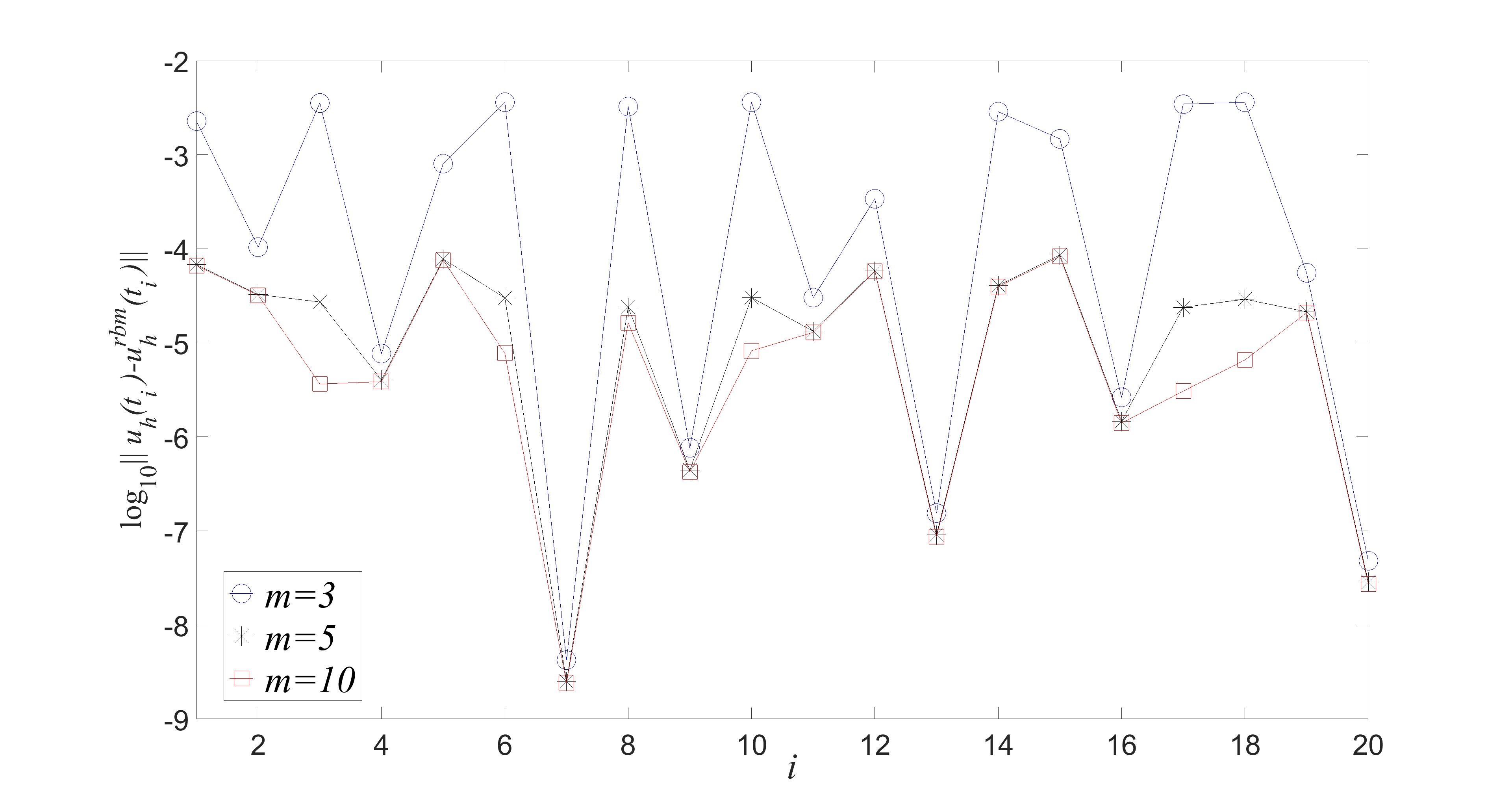}
    \caption{RCGBM error at $t=t_i$ with 2146689 grid vertices}
    \label{fig:errorCube}
  \end{figure}
\begin{table}[H]
  \centering
  \begin{tabular}{ |r||p{1.7cm}|p{1.5cm}| p{2.3cm}|p{2.0cm}| }
    \hline
    \hline
    $NV$ & $\|u-u_h^{\rm rbm}\|$ & $\|u-u_h\|$ & $\|\nabla(u_I-u_h^{\rm rbm})\|$ & $\|\nabla(u_I-u_h)\|$  \\
    \hline
    4913     & 4.3253e-3 & 4.3032e-3  &~~2.4902e-3 &~~2.5969e-3\\
    35937     & 1.0825e-3 &  1.0772e-3 &~~6.2538e-4 &~~7.4115e-4\\
    274625     & 2.7065e-4 &  2.6941e-4 &~~1.5831e-4 &~~2.0038e-4\\
    2146689     & 6.7759e-5 & 6.7460e-5  &~~4.0662e-5 &~~5.2585e-5\\
    \hline
   \end{tabular}
   \caption{Errors of the RCGBM with $m=10$ and full rational approximation on the cube.\label{tab:errorCube}}
   \end{table}
  
\subsection{Fractional Laplace-Beltrami Equation}\label{subsec:surface}
In the second example, we set $\mathcal{A}$ to be the negative Laplace-Beltrami operator $-\Delta_{\Omega}$, i.e., the surface Lapalcian, on the 2d-sphere 
\[\Omega=\big\{(x,y,z)\in\mathbb{R}^3: x^2+y^2+z^2=1\big\}.
\]
We use the eigenfunction of $-\Delta_\Omega$ as the exact solution 
$$u(x,y,z)=x+2y+3z,$$
and the corresponding right hand side is
$$f=(-\Delta_\Omega)^\frac{1}{2}u=2^\frac{1}{2}u.$$
There is no boundary condition for the problem defined $\Omega$ and the uniqueness of the solution is guaranteed by zero mean $\int_\Omega udS=0$.
The initial mesh  for $\Omega$ is an inscribed octahedron, which is further refined by the uniform red refinement. In the refinement process, the newly added nodes are normalized to be aligned with the unit sphere. When solving $(-\Delta_\Omega)^\frac{1}{2}u=f$, we employ the linear nodal element space $\mathcal{H}_h$ on the triangulated surface, see, e.g., \cite{Dziuk1988}.

For general surfaces, it is inconvenient to construct grid hierarchy and multigrid based on algebraic coarsening is preferred in such cases (cf.~\cite{Li2021SISC,Li2023FoCM}). To show the applicability of our RCGBM on general surfaces, we use the classical algebraic multigrid with W-cycles as the preconditioner $\mathbb{B}$ in Algorithm \ref{a:RCGBMmatrix}. We test RCGBM(1,2,$m$) on the surface because of the one-dimensional kernel of  
$\mathbb{A}$. The eigenvalue upper bound is set to be $\Lambda=14/h^2$, where $h^2$ is the area of the triangle in the current mesh. The numerical solution $u_h^{{\rm rbm}}$ is shifted to satisfy the zero mean constaint. Numerical results are demonstrated in Figure \ref{fig:errorSurface}
and Table \ref{tab:errorSurface}.

 \begin{figure}[H]
    \centering
\includegraphics[width=10cm,height=6cm]{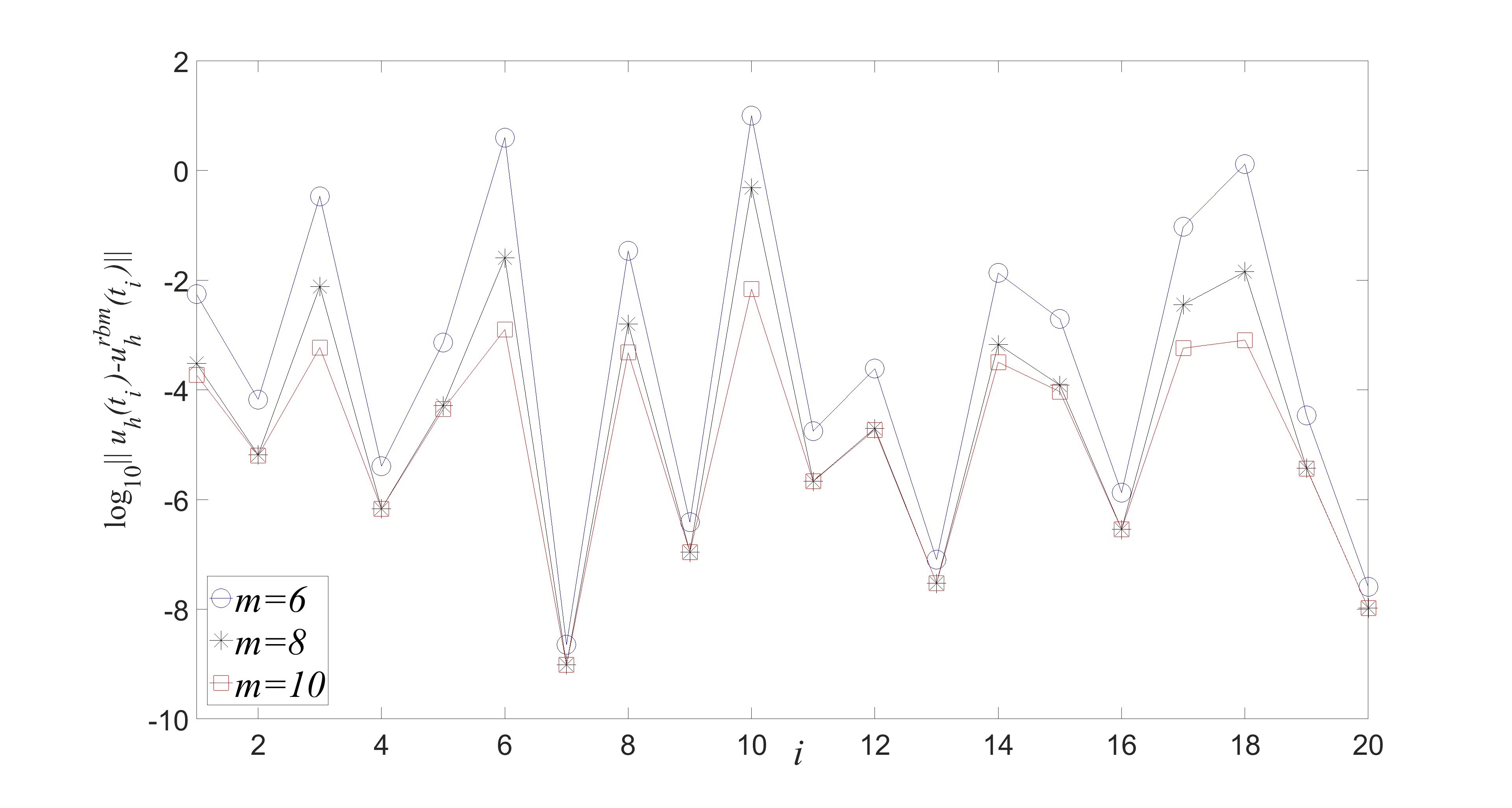}
    \caption{RCGBM error at $t=t_i$ with 1048678 grid vertices}
    \label{fig:errorSurface}
  \end{figure}

\begin{table}[H]
  \centering
  \begin{tabular}{ |r||p{2.8cm}|p{2.8cm}|p{3.0cm}| }
    \hline
    \hline
    $NV$ & $\|u-u_h^{\rm rbm}\|$, $m=6$ & $\|u-u_h^{\rm rbm}\|$, $m=8$ & $\|u-u_h^{\rm rbm}\|$, $m=10$ \\
    \hline
    4098     &~~~~6.3355e-3 &~~~~6.3284e-3  &~~~~~6.3280e-3 \\
    16386     &~~~~1.6530e-3 &~~~~1.5828e-3 &~~~~~1.5827e-3\\
    65538     &~~~~1.3436e-3 &~~~~3.9507e-4 &~~~~~3.9501e-4 \\
    262146     &~~~~1.3546e-3 &~~~~9.7446e-5  &~~~~~9.7029e-5 \\
    1048678     &~~~~1.3261e-3 &~~~~1.2277e-4  &~~~~~6.3104e-6\\
    \hline
   \end{tabular}
   \caption{Errors of the RCGBM with $m=6, 8, 10$ on the sphere.\label{tab:errorSurface}}
   \end{table}
As shown in Figure \ref{fig:errorCube}, the error $\|u_h(t_i)-u^{{\rm rbm}}_h(t_i)\|_{L_2(\Omega)}$ decays as the number $m$ of conjugate gradient basis grows. However, it is observed that $m=6, 8$ is not enough to sufficiently reduce the $\|u_h(t_i)-u^{{\rm rbm}}_h(t_i)\|_{L_2(\Omega)}$ for several poles. The reason is that the AMG is not as accurate as the geometric multigrid used in Subsection \ref{subsec:cube}. The errors in Table \ref{tab:errorCube} show that the overall $L_2$ convergence of the RCGBM solution $u_h^{{\rm rbm}}$ with $m=10$ is quite satisfying, while RCGBM(1,2,6) and RCGBM(1,2,8) stop converging at certain scales, suggesting that a slightly larger $m$ is needed for RCGBMs based on user-friendly preconditioners such as AMGs that are not of very high quality. 

\subsection{Fractional Graph Laplacian}

In the third experiment, we study the performance of the RCGBM on a connected graph $\Omega$ with $n$ vertices. We employ the function  \texttt{sprand} in \texttt{MATLAB} with density $5/n$ to generate a sparse and positive semi-definite random matrix $\mathbb{L}$, which corresponds to the graph Laplacian of a graph with  randomly assigned weights. Let $\mathbb{I}$ be the identity matrix, $\mathcal{A}_h=\mathbb{A} = \mathbb{L} + \mathbb{I}/n$, and $f_h=\mathbf{f}$ be generated by \texttt{randn}. The RCGBM(0,1,$m$) with $\mathbb{B}=\mathbb{A}^{-1}$ is applied to the rescaled problem \eqref{rescaled} with $\Lambda=\|\mathbb{A}\|_{\infty}$.
Numerical results are exhibited in Figure \ref{fig:errorGL} and Table \ref{tab:errorGL}.

In this example, we do not know the exact solution $\mathbb{A}^{-\frac{1}{2}}\mathbf{f}$ and thus only compute errors between the high-fidelity solution $\mathbf{u}(t_i)$ and the reduced basis solution $\mathbf{u}^{\rm rbm}(t_i)$ measured by the $L_2$ norm $\|\cdot\|$ supported at graph vertices.  The relative error $\|\mathbf{u}(t_i)-\mathbf{u}^{{\rm rbm}}(t_i)\|/\|\mathbf{u}(t_i)\|$ still diminishes as the number of reduced basis elements increases, see Figure \ref{fig:errorGL}. The  results in Table \ref{tab:errorGL}  demonstrate that the RCGBM remains stable regardless of the problem's scale, and employing $m = 8$ reduced basis yields a satisfactory approximation.
  \begin{figure}[H]
    \centering
\includegraphics[width=10cm,height=6cm]{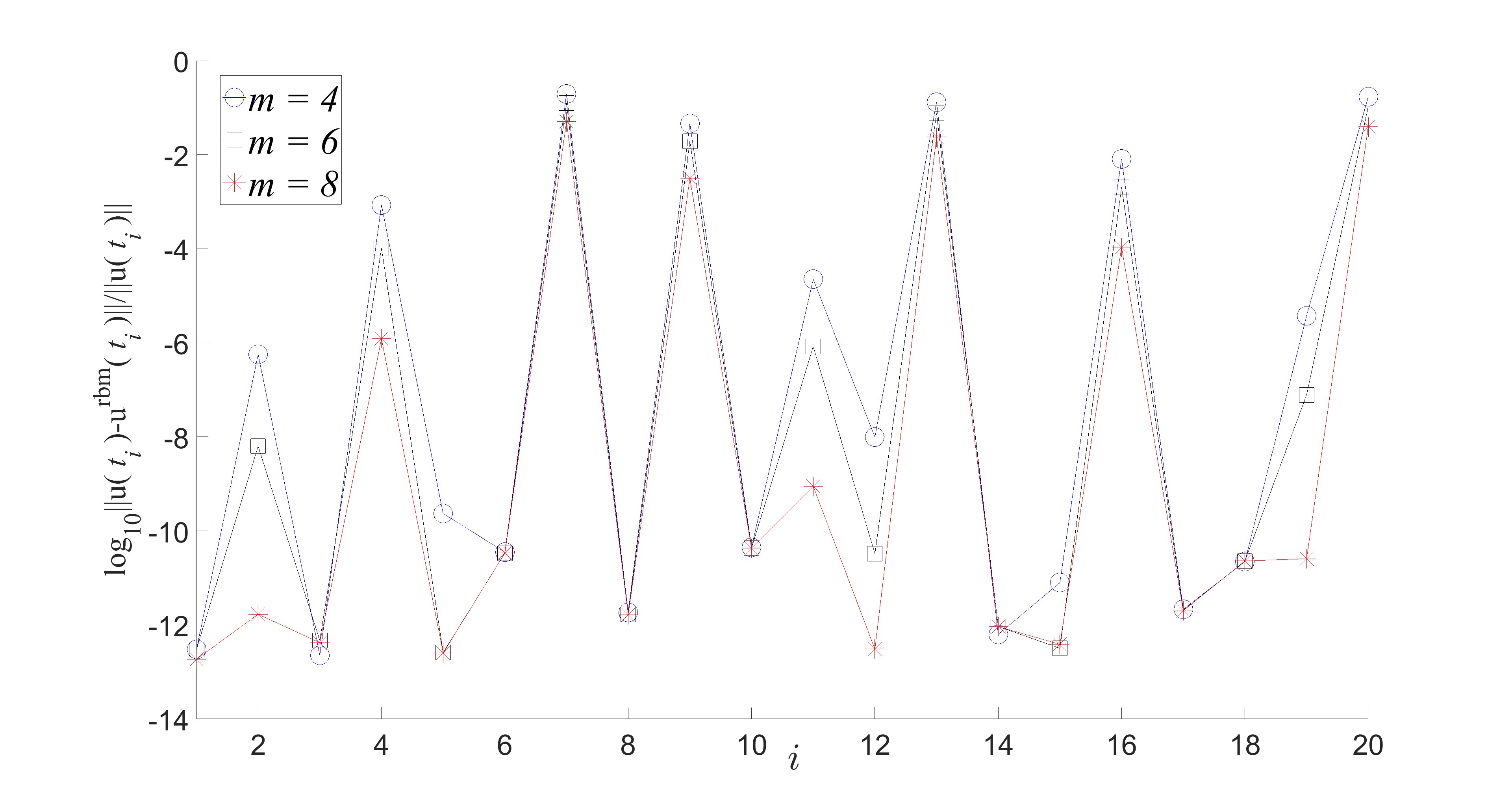}
    \caption{RCGBM relative errors at $t=t_i$ with $n = 262144$}
    \label{fig:errorGL}
  \end{figure}
  

\begin{table}[H]
  \centering
  \begin{tabular}{ |r||p{1.5cm}|p{1.5cm}| p{1.5cm}| }
    \hline 
    \multicolumn{4}{|c|}{\rule{0pt}{10pt}$\|\mathbf{u}-\mathbf{u}^{{\rm rbm}}\|/\|\mathbf{u}\|$} \\
    \hline
    $n$ & $m = 4$ & $m = 6$ & $m = 8$ \\
    \hline
    4096 & 2.6078e-2 & 7.2151e-3 & 3.8439e-3 \\
    16384 & 5.2934e-2 & 1.4381e-2 & 7.8110e-3 \\
    65536 & 2.1039e-2 & 6.1611e-3 & 3.3951e-3 \\
    262144 & 2.2246e-2 & 1.2911e-2 & 4.3708e-3 \\
    \hline
   \end{tabular}
   \caption{Relative errors of the RCGBM with m = 4, 6, 8}
   \label{tab:errorGL}
\end{table}

For the graph Laplacian, the error of the RCGBM is relatively large  for poles of magnitude $\mathcal{O}(1)$. In view of Theorem \ref{errorRCGBM}, the reason is that $\kappa\big(\mathbb{B}(\mathbb{A}+t_i\mathbb{I})\big)$ is not well bounded for $t_i=\mathcal{O}(1)$. To improve the accuracy in such cases, we use the RCGBM(0,1,$m$) with $\mathbb{B}=\mathbb{A}^{-1}$ to solve the linear systems corresponding to the first 10 smallest $t_i$ and employ another RCGBM(0,1,$m$) with the preconditioner $\mathbb{B}=(\mathbb{A} + t_{\max}\mathbb{I})^{-1}$ to solve the rest of 10 large-scale systems, where $t_{\max}$ is the $t_i$ of the largest modulus. It can be observed from Figure \ref{fig:errorGL_2preconds} and Table \ref{tab:errorGL_2preconds} that this strategy significantly decreases the error of the RCGBM.

  \begin{figure}[H]
    \centering
\includegraphics[width=10cm,height=6cm]{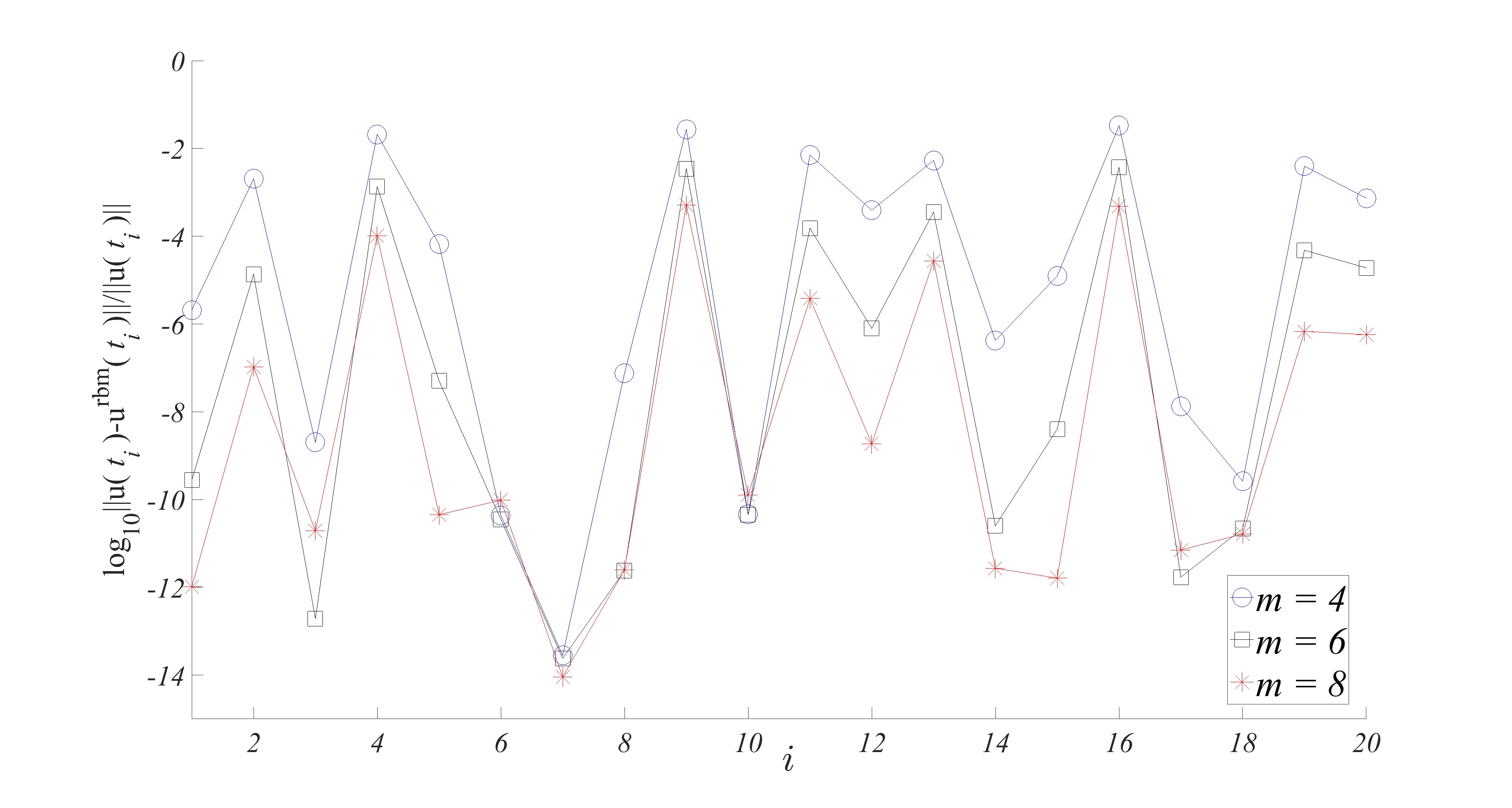}
    \caption{RCGBM relative errors at $t=t_i$ with $n = 262144$ and two preconditioners}
    \label{fig:errorGL_2preconds}
  \end{figure}

\begin{table}[H]
  \centering
  \begin{tabular}{ |r||p{1.5cm}|p{1.5cm}| p{1.5cm}| }
    \hline 
    \multicolumn{4}{|c|}{\rule{0pt}{10pt}$\|\mathbf{u}-\mathbf{u}^{{\rm rbm}}\|/\|\mathbf{u}\|$} \\
    \hline
    $n$ & $m = 4$ & $m = 6$ & $m = 8$ \\
    \hline
    4096 & 6.5604e-3 & 5.4840e-4 & 5.1519e-5 \\
    16384 & 1.2867e-2 & 1.1082e-3 & 1.1291e-4 \\
    65536 & 4.9327e-3 & 4.3909e-4 & 4.8909e-5 \\
    262144 & 4.7487e-3 & 4.8781e-4 & 6.3574e-5 \\
    \hline
   \end{tabular}
   \caption{Relative errors of the RCGBM with m = 4, 6, 8 and two preconditioners}
   \label{tab:errorGL_2preconds}
\end{table}

\section{Conclusions}\label{sec:con}
We have designed an OGA for rational approximation of fractional power functions and an efficient RCGBM for approximately solving discrete fractional diffusion problems. It is straightforward to extend the  RCGBM proposed here to more general coercive elliptic PDEs such as a family of second order elliptic PDEs with affinely parametrized diffusion coefficients. Along the same line, the design of accurate reduced basis methods based on Krylov space methods for indefinite and non-symmetric parameter-dependent problems is a subject of current and future research.

\section*{Acknowledgements}
The authors wish to thank Professor Jinchao Xu for the discussion on reduced basis methods and applications of the PCG in Hilbert spaces. Li  would also like to thank Aidi Li for an earlier version of the OGA code.

\appendix
\section{The OGA and numerical PDE codes}
In the appendix, we present the  \texttt{MATLAB}/\texttt{GNU Octave} code of the orthogonal greedy algorithm \ref{a:OGA} in Listing 1. This code is used for generating the rational approximant in Table \ref{tab:OGArespol}, which contains the residues $c_i$ and poles $-t_i$ used in the RCGBM($s,d,m$) in Section \ref{sec:num}. The finite element codes for solving fractional diffusion are posted in  \href{https://github.com/yuwenli925/RCGBM}{github.com/yuwenli925/RCGBM}. 

\begin{table}[H]
  \centering
  \begin{tabular}{ |p{0.4cm}||p{3.1cm}|p{3.1cm}| }
    \hline
    \hline
    $i$  & ~~~~~~~~~~~$c_i$ & ~~~~~~~~~~~~$t_i$  \\
    \hline
    1     & 0.001163181622200
   &  0.000018062500000 \\
    2     & 0.013218620357233
     &  0.000933302500000 \\
    3     & 0.000255953763580
    &   0.000000562500000\\
    4     & 0.046695455564245
    &  0.013156090000000 \\
    5     & 0.002833475749545
    &  0.000104040000000 \\
    6     & 0.000144842262876
    &  0.000000062500000 \\
    7     & 7.333904150551077
    &  25.000000000000000 \\
    8     & 0.000535823283956
    &  0.000003422500000 \\
    9     & 0.176401964217136
    &  0.136641122500000 \\
    10     & 0.000085661432787
    &  0.000000002500000 \\
    11     & 0.026987066006317
    &  0.003271840000000 \\
    12     & 0.005823374267035
    &  0.000277222500000 \\
    13     & 0.418860226358461
    &  0.684011702500000 \\
    14     & 0.000717641159456
    &  0.000008122500000 \\
    15     & 0.001870543346405
    &  0.000044222500000 \\
    16     & 0.063983537529953
    &     0.039243610000000 \\
    17     & 0.000318817071900
    & 0.000001440000000\\
    18     & 0.000126652196126
    &  0.000000202500000 \\
    19     & -0.002953457313354
    &  0.001759802500000 \\
    20     & 0.494928246538564
    &  2.185962250000000 \\
    \hline
  \end{tabular}
    \caption{Residues and poles of the OGA rational approximant on $[10^{-6},1]$.\label{tab:OGArespol}}
  \end{table}

\begin{lstlisting}[caption={\rm OGA code}]
function [res, pol, err] = OGA(s,iter)
%approximate z.^s on [1e-6,1] with iter OGA iterations
f = @(z) z.^(-s); epsi = 1e-8; 
node = unique([epsi:(0.001-b1)/2000:0.001,...
 0.001:(0.01-0.001)/2000:0.01,0.01:(1-0.01)/3000:1]');
h = node(2:end) - node(1:end-1); %graded mesh on [epsilon,1]
c = [1/2-sqrt(15)/10 1/2 1/2+sqrt(15)/10]; %Gauss quadrature
qpt = [node(1:end-1)+c(1)*h;node(1:end-1)+c(2)*h;...
node(1:end-1)+c(3)*h];
hd = 5e-5;t = (hd:hd:5)';t = t.^2;nd = length(t); %discrete dictionary
normg = sqrt(1./(epsi+t)-1./(1+t)); %L2 norm of dictionary elements
g = 1./(repmat(qpt,1,nd)+t')./normg'; %normalizing dictionary
fqpt = f(qpt); r = fqpt;
A = zeros(iter,iter); rhs = zeros(iter,1);
id = zeros(iter,1); argmax = zeros(nd,1);
for i = 1:iter
    for j = 1:nd
       argmax(j) = product(g(:,j),r,h);
    end
    [~,id(i)] = max(abs(argmax));
    for j = 1:i
        A(j,i) = product(g(:,id(j)),g(:,id(i)),h);
        A(i,j) = A(j,i);
    end
    rhs(i) = product(g(:,id(i)),fqpt,h);
    C = lsqminnorm(A(1:i,1:i),rhs(1:i));
    r = fqpt - g(:,id(1:i))*C;
    fprintf("Step \%d\n",i);
end
res = C./normg(id); %residues of the approximant
pol = t(id); %(-1)*poles of the approximant
maxnode = linspace(1e-6,1,5000000)';
rf = zeros(length(maxnode),1);
for j=1:iter
    rf = rf + res(j)./(maxnode+pol(j));
end
err = max(abs(rf - f(maxnode))); %max norm error on [1e-6,1]
function z = product(f1,f2,h)
z = 5/18*sum( h.*f1(1:3:end).*f2(1:3:end) )+...
    4/9*sum( h.*f1(2:3:end).*f2(2:3:end) )+...
    5/18*sum( h.*f1(3:3:end).*f2(3:3:end) );
end
end
\end{lstlisting}

\bibliographystyle{siamplain}

\end{document}